\documentclass{article}

\usepackage{graphicx}%
\usepackage{amsmath,amssymb,amsfonts}%
\usepackage{amsthm}%
\usepackage{mathrsfs}%

\newtheorem{theorem}{Theorem}[section]
\newtheorem{lemma}[theorem]{Lemma}
\newtheorem{definition}[theorem]{Definition}

\newtheorem{proposition}[theorem]{Proposition}
\newtheorem{remark}[theorem]{Remark}

\newtheorem{example}[theorem]{Example}

\newcommand{\Matrix}[1]{\ensuremath{\left[\begin{array}{ccccccccccccccccccccccccr} #1 \end{array}\right]}}
\newcommand{\ii}{\mathrm{i}}
\newcommand{\Z}{\mathbb{Z}}
\newcommand{\R}{\mathbb{R}}

\newcommand{\D}{\mathbb{D}}

\newcommand{\Sone}{{\bf S}^1}
\newcommand{\ee}{{\mathrm e}}
\newcommand{\beqn}{\begin{eqnarray*}}
\newcommand{\eeqn}{\end{eqnarray*}}

\newcommand{\C}{\mathbb{C}}
\newcommand{\DD}{{\mathrm D}}
\newcommand{\ONE}{{\bf 1}}

\newcommand{\dd}{\mathrm{d}}

\newcommand{\eps}{\varepsilon}

\title{Hopf Bifurcation in Asymmetric Ring Networks: Constraints on Phase Shifts}
\author{Ian Stewart \\ Mathematics Institute
\\ University of Warwick \\ Coventry CV4 7AL
\\ United Kingdom}

\begin{document}

\maketitle

\begin{abstract}
Hopf bifurcation in networks of coupled ODEs
creates periodic states in which the relative phases of
nodes are well defined near bifurcation. When the network is a
fully inhomogeneous nearest-neighbour coupled unidirectional ring,
 and node spaces are 1-dimensional, we derive constraints on these phase shifts that
apply to any ODE that respects the ring topology.
We begin with a 3-node ring and generalise the results
to any number of nodes. The main point is that such constraints exist
even when the only structure present is the network topology.
We also prove that the usual nondegeneracy conditions in the classical
Hopf Bifurcation Theorem are valid generically for ring networks, by perturbing only
coupling terms.
\end{abstract}

\maketitle

\section{Introduction}

Hopf bifurcation is a mechanism by which a steady state of a family of ODEs
becomes unstable and throws
off a periodic cycle \cite{GH83,HKW81}. It occurs when the linearised ODE has purely imaginary eigenvalues,
subject to various nondegeneracy conditions, namely: 
a simple pair of imaginary eigenvalues, nonresonance, and
the eigenvalue crossing condition. Other variations on these conditions are
common in the literature. In particular, nonresonance is usually
replaced by the stronger condition  `no other imaginary eigenvalues'. See Section \ref{S:HBT}.

Periodic states of uni- or bidirectional
rings of identical oscillators have been widely studied; 
see for example\cite{A86,E85}.
Many other references, often using specific oscillator equations
such as van der Pol oscillators, are listed in \cite{S23a}. The role of
the symmetry group of the ring 
 has been made explicit in
\cite{GS86,GS02,GSS88,S23a} using symmetric (or equivariant)
Hopf bifurcation for $n$-node rings of coupled dynamical systems
with cyclic symmetry group $\Z_n$ or dihedral symmetry group $\D_n$. 
These works classify the typical patterns of phase relations that
arise via Hopf bifurcation in such networks, and gives conditions for them to occur.
The underlying ODE is assumed to
be `admissible'; that is, consistent with both the network topology and
the group of symmetries. 
More general network topologies have also been studied in this manner.

In this paper we consider what happens when this symmetry constraint is removed;
that is, when the nodes and arrows of the ring network all have different types. We work in
 the general formalism for network dynamics of \cite{GS23,GST05,SGP03}.
Such networks are said to be {\em fully inhomogeneous}. The symmetry group
is now trivial, and
the definition of an admissible ODE is straightforward: the component
of the ODE for node $c$ must have the form
\[
\dot x_c = f_c(x_c, x_{i_1}, \ldots, x_{i_m})
\]
where $x_i$ is the variable associated with node $i$, the nodes
$i_1, \ldots, i_m$ are the tail nodes of the input arrows to node $c$,
and the $f_c$ are independent functions as $c$ runs over the set of nodes.

Bifurcations in fully inhomogeneous networks have been
studied in \cite{GGPSW19} in the case of certain mode interactions.
It has been proved in \cite{J12,S20rigideq} that
fully synchronous equilibrium states, in which all node variables have the same 
value, cannot occur rigidly. That is, synchrony
cannot persist under small admissible perturbations of a hyperbolic
equilibrium. (However, special constraints, the commonest being to require
$f_c(0,0,\ldots, 0) = 0$ for all $c$, can impose this kind of synchrony.)
Discrete dynamics of coupled map networks is discussed in \cite{AJH05,JAH05,KY10}.
These works aside,
there seem to be few general results about fully inhomogeneous networks.

Here we show that, for any fully inhomogeneous ring network
with 1-dimen\-si\-on\-al node spaces and nearest-neighbour unidirectional coupling,
there are `universal' constraints on relations between the phases
of successive nodes in any periodic state arising via Hopf bifurcation.
We make no synchrony assumptions on the family of equilibria concerned.
The phases are defined at the bifurcation point by the phase
relations of the linearised eigenfunction, and remain approximately valid
sufficiently close to the bifurcation point. We also show that
for these networks the nondegeneracy conditions in the
Hopf Bifurcation Theorem are generic; that is, can be realised
after an arbitrarily small {\em admissible} perturbation of the ODE.
(Section \ref{S:HBT} explains why this statement is not obvious.)
We do not address stability, but the absence of linear degeneracy implies that usual 
exchange of a stability criterion for supercritical and subcritical branches
applies \cite[Chapter 1 Section 4]{HKW81}. 
We expect the necessary cubic order term to be nonzero generically
for a ring network, but have not attempted to prove this.

Even for symmetric rings, it is shown in \cite{S23a} that
such constraints do not apply if longer-range couplings are present
or the node state spaces have higher dimension. These negative results
also apply to fully inhomogeneous networks: just perturb the ODE 
(admissibly) to break the symmetry.

\subsection{Summary of Paper}

Section \ref{S:HBT} reviews the classical Hopf Bifurcation Theorem
and several variants, including one that requires no resonant imaginary
eigenvalues. We point out that in equivariant Hopf bifurcation
such resonances are non-generic, but give an example to show
that in network admissible ODEs this statement can be false.

Section \ref{S:3NFIUR} considers a special but typical example: 
a 3-node network, nearest-neighbour coupled in a unidirectional ring,
with $1$-dimensional node spaces. We
establish explicit conditions for the occurrence of a Hopf bifurcation, 
find the eigenvalues and eigenvectors of the Jacobian, and
observe that relative phase shifts make sense near a Hopf bifurcation point
even when the waveforms being compared are not identical up to time translation.
We compute the phase shifts between successive nodes and
classify the constraints on the quadrants of the unit circle in $\C$
in which they can lie, assuming that the period has been normalised to $2\pi$
by scaling time, and the phases $\theta$ are represented by points
$\ee^{\ii\theta}$ on the unit circle.

Section \ref{S:DRnN} generalises some of these results to $n$-node nearest-neighbour coupled rings.
We prove that simple eigenvalues are generic and nonresonance is generic,
so these hypotheses
of the classical Hopf Bifurcation Theorem are generically valid.
Indeed, only the coupling terms need to be perturbed to ensure simplicity of the imaginary eigenvalue:
the internal dynamics of the nodes can remain unchanged.
The remaining nondegeneracy hypothesis is the eigenvalue crossing
condition, which is obviously generic (if eigenvalues do not
cross the imaginary axis with nonzero speed,
add a term $\alpha x$
for small $\alpha$ to make the eigenvalues cross the imaginary axis with
speed $\alpha$). Thus classical Hopf bifurcation is generic in 
fully inhomogeneous ring networks.

\section{Hopf Bifurcation Theorem}
\label{S:HBT}

Hopf's original bifurcation theorem \cite{H42} has since been generalised in several
ways. He assumed the vector field is analytic; this can be weakened to
$C^\infty$ or even $C^4$; see \cite[Chapter 1 Section 2 Theorem II]{HKW81}.
It is stated in various forms, some more restrictive than others; stronger
hypotheses lead to stronger conclusions.

We state it in the following form:

\begin{theorem}
\label{T:hopf_bif}
Let $f:\R^n\times\R \to \R^n$ be a $C^\infty$ map with a local branch of equilibria $(x(\lambda),\lambda)$; that is,
$f(x(\lambda),\lambda) = 0$ for all $\lambda$ near some point 
$\lambda_0$. Let $x_0=x(\lambda_0)$. Consider the family of
ODEs $\dot x = f(x,\lambda)$ for $(x, \lambda)$ near $(x_0,\lambda_0)$. Suppose
that the derivative $\mathrm{D} f|_{(x_0,\lambda_0)}$ has a conjugate pair of complex
eigenvalues $\sigma(\lambda) + \ii \rho(\lambda)$ such that
$\sigma(\lambda_0) = 0$. Let $\rho(\lambda_0) = \omega \neq 0$ and assume
the following nondegeneracy conditions:

{\rm (1)} {\it Eigenvalue Crossing Condition}: $\frac{\dd}{\dd \lambda}\rho(\lambda) \neq 0$
when $\lambda = \lambda_0$;

{\rm (2)} {\it Simple Eigenvalue Condition}: The eigenvalues $\pm \ii\omega$
of $\DD f|{(x_0,\lambda_0)}$ are simple;

{\rm (3)} {\it Nonresonance Condition}: $\mathrm{D} f|_{(x_0,\lambda_0)}$ has no purely imaginary
eigenvalues $\pm k\ii\omega$ for integer $k> 1$.
%\footnote{Can't find this stated anywhere but
%there was an issue with Bill Langford's paper...}

Then the branch of equilibria $(x(\lambda),\lambda)$ bifurcates
to a branch of periodic states at $\lambda = \lambda_0$. The period tends to
$T = \frac{2\pi}{|\omega|}$ as $\lambda \to \lambda_0$.
\end{theorem}

(After a $\lambda$-dependent translation of
 coordinates we may assume that $\lambda_0 = 0$ and $x(\lambda) = 0$
for $\lambda$ near $\lambda_0$.)

In \cite{H42} and \cite{HKW81} condition (3) is replaced by `the remaining eigenvalues
have strictly negative real parts'. In particular, this condition implies that there are no other
eigenvalues on the imaginary axis. This condition is necessary (though not always sufficient) for
a stable periodic branch to occur, but it can be weakened if only the {\em existence}
 of the branch
is being proved. It is enough for the remaining eigenvalues to
lie off the imaginary axis; see \cite{H93,HMO84,HMO02}. 
But even this condition can be weakened, as we now explain.

The method of \cite{H93,HMO84} replaces the ODE by an operator
equation on `loop space', the space of $T$-periodic functions $\R \to \R^n$ with
a suitable norm. It then uses Liapunov-Schmidt reduction to define a
`reduced function' from the kernel of the operator to the range, whose
zeros are in one-to-one correspondence with periodic states of the ODE.
The kernel and range can be identified and are finite-dimensional.
In this approach, when $T = 2\pi/|\omega|$, 
the kernel is the sum of the imaginary eigenspaces $E_{\pm k\ii\omega}$
for integers $k$. (This sum is needed because a state with period $T/k$ also has 
period $T$.)
Condition (3) ensures that the kernel is just $E_{\pm\ii\omega}$, which
is 2-dimensional by condition (2). It is then
possible to prove that the reduced equation then has solution branches
provided condition (1) holds. Indeed, this version follows from
the Equivariant Hopf Theorem of \cite{GS85a,GSS88} when the symmetry group
of the ODE is the trivial group $\ONE$, because the circle group of phase shifts
splits off the resonant eigenvalues.

More generally, the same loop space approach is used to prove the Equivariant Hopf
Theorem of \cite{GS85a,GSS88}. There is is proved that in the presence
of a symmetry group $\Gamma$, the imaginary eigenspace is
generically a $\Gamma$-simple representation. This is the analogue
for equivariant ODEs of the condition that the remaining eigenvalues
should be off the imaginary axis. The proof that this condition is generic involves perturbing the
ODE by a small scalar multiple of the projection map onto any other
irreducible component.
Again, the nonresonance condition (3) ensures that the kernel is 
just $E_{\pm\ii\omega}$, which leads to a   
version of the Equivariant Hopf Theorem that is slightly more general
than the one stated in \cite{GS85a,GSS88}, in which 
the absence of other imaginary eigenvalues is replaced by the
absence of other $k:1$ resonant imaginary eigenvalues.
Indeed, a further refinement is also possible: if a subgroup 
$\Sigma \subseteq \Gamma \times \Sone$ has a
$2$-dimensional fixed-point subspace when restricted to the space of
eigenvectors $\sum_{k \in \Z} E_{\pm k \ii \omega}$, and this
subspace is contained in $E_{\pm \ii \omega}$, then the conclusions
of the Equivariant Hopf Theorem remain valid, with essentially the same proof.

However, this technique need not be
valid for an admissible ODE of a $\Gamma$-symmetric network. The reason is that
equivariant maps for a network with symmetry group $\Gamma$ 
need not be admissible. Specifically, the projection map onto an
 irreducible component, which is crucial to the proof of the
basic theorem that
generically the critical eigenspace for
an imaginary eigenvalue is $\Gamma$-simple,
need not be admissible for the network. 
This remark applies even when the symmetry group is the trivial group $\ONE$,
and it is why we impose the nonresonance condition (3) as well as (2).

We give two examples  to illustrate the need for the 
nonresonance condition. Both of them use 
homogeneous (indeed regular) networks, where the admissibility condition is 
much stronger than in the inhomogeneous case. In contrast, Theorem \ref{T:nores}
below shows that resonances in fully inhomogeneous $n$-node 
unidirectional rings with
nearest-neighbour connections can be removed by admissible perturbations.

\begin{example}\em
\label{ex:5node_resonant}
Figure \ref{F:5node_resonant} (left) shows a regular transitive 5-node network
of valence (or in-degree) $3$; that is, each node has three input arrows.
`Transitive' means that any two nodes are joined by a directed path \cite{AGU72}.
 Its adjacency matrix is 
\[
A = \Matrix{1 & 1& 0&1&0 \\ 1&1&0&0&1 \\ 0&2&0&0&1 \\ 0&1&1&0&1 \\
1& 0& 1& 0& 1}
\]
with the convention that the entry $a_{ij}$ is the
number of arrows from node $j$ to node $i$. (In the graph theory literature
it is usual to define the adjacency matrix so that $a_{ij}$ is the
number of arrows from node $i$ to node $j$.)
The eigenvalues of $A$ (with either convention) are $3, \ii, -\ii, 0, 0$. 
The family of linear admissible ODEs
\[
\dot x = (\lambda I + J)x
\]
has a $0:1$ resonance at $\lambda = 0$ that cannot be removed by any
 admissible perturbation.

Zero eigenvalues are normally associated with steady state branches,
but they still count as resonances in this context because they
also contribute to the kernel of the operator.

\begin{figure}[h!]
\centerline{%
\includegraphics[width=0.3\textwidth]{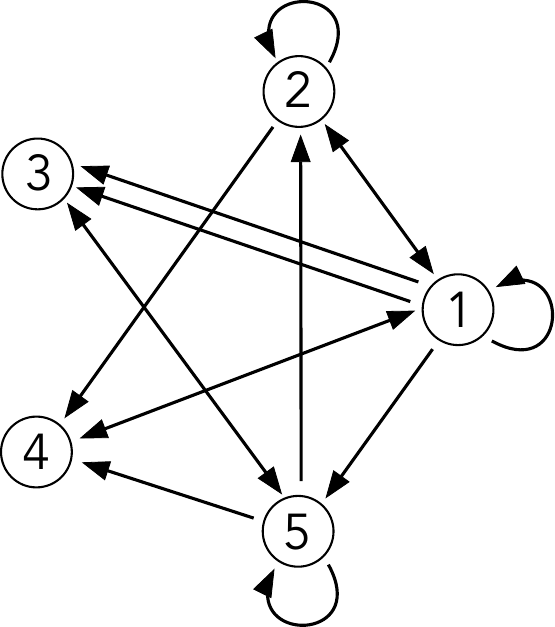}\qquad
\includegraphics[width=0.35\textwidth]{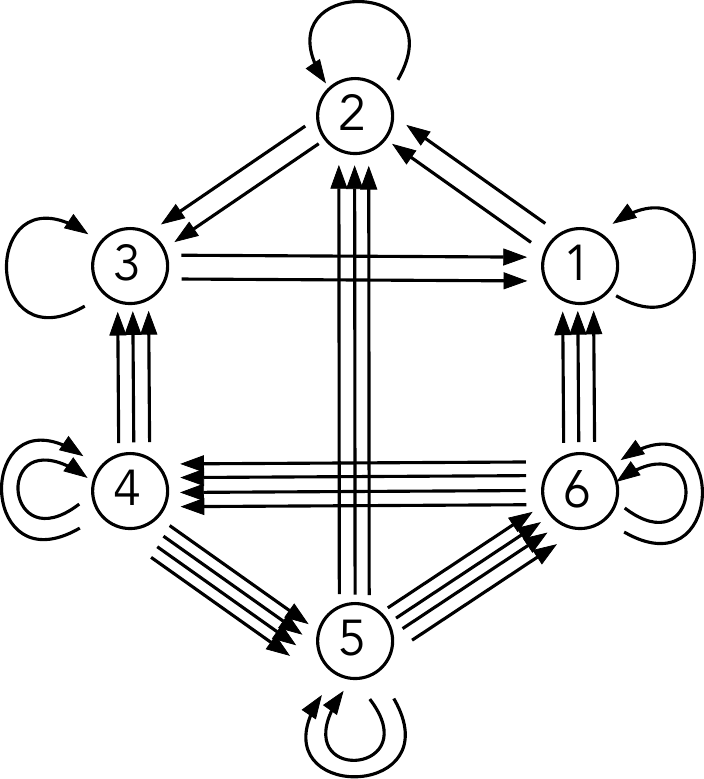}
}
\caption{{\em Left}: Regular transitive 5-node network with $0:1$ resonance.
{\em Right}: Regular 6-node network with $2:1$ resonance.}
\label{F:5node_resonant}
\end{figure}

Figure \ref{F:5node_resonant} (right) shows a regular 6-node network
of valence (or in-degree) $6$. It is not transitive: nodes 
$1,2,3$ have inputs from nodes $4,5,6$, but 
nodes $4,5,6$ have no inputs from nodes $1,2,3$.
It has two transitive components $\{1,2,3\}$ and $\{4,5,6\}$.
Its adjacency matrix is 
\[
A = \Matrix{1 & 0& 2&0&0& 3\\ 2&1&0&0&3 &0 \\ 0&2&1&3&0 & 0 \\ 0&0&0&2&0&4 \\
0&0&0&4&2&0 \\ 0&0&0&0&4&2
}
\]
The eigenvalues are $3,6, \pm \ii\sqrt{3}, \pm 2\ii\sqrt{3}$,
a $2:1$ resonance. Again, no admissible perturbation can
remove this resonance.
\end{example}

\section{3-Node Fully Inhomogeneous Unidirectional Ring}
\label{S:3NFIUR}
 
We begin with a simple example: a 3-node network coupled in 
a unidirectional ring, Figure \ref{F:3ring_inhomog}. We assume that
each node has a 1-dimensional state space $\R$. In Section \ref{S:DRnN} we generalise
the results to $n$-node rings, still with 1-dimensional node spaces.

\begin{figure}[htb]
\centerline{\includegraphics[width=0.2\textwidth]{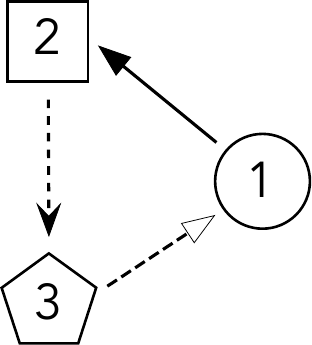}
}
\caption{3-node fully inhomogeneous unidirectional ring.}
\label{F:3ring_inhomog}
\end{figure}

Admissible ODEs for this network take the general form
\beqn
\dot{x_1} &=& f_1(x_1, x_2) \\
\dot{x_2} &=& f_2(x_2, x_3) \\
\dot{x_3} &=& f_3(x_3, x_1) 
\eeqn
and the Jacobian (evaluated at any point) therefore has the form
\begin{equation}
\label{e:jac}
J = \Matrix{a_1&b_1&0\\0&a_2&b_2\\b_3&0&a_3}
\end{equation}
where
\beqn
&& a_1 = \partial_1 f_1(x_1,x_2) \qquad a_2 = \partial_1 f_2(x_2,x_3) \qquad a_3 = \partial_1 f_3(x_3,x_1) \\
&& b_1 = \partial_2 f_1(x_1,x_2) \qquad b_2 = \partial_2 f_2(x_2,x_3) \qquad b_3 = \partial_2 f_3(x_3,x_1) 
\eeqn
and $\partial_i$ indicates the partial derivative with respect to the $i$th variable.

\subsection{Hopf Conditions}

Suppose there is a Hopf bifurcation at some point.
If so, the eigenvalues of $J$ in~\eqref{e:jac}, evaluated at
that point, are $\tau, i \omega, -\ii \omega$,
where $\tau, \omega \in \R$ and $\omega > 0$.
Take the trace:
\begin{equation}
\label{E:trace}
\tau = a_1+a_2+a_3
\end{equation}
There are two obvious expressions for (minus) the characteristic polynomial:
\beqn
(x^2+\omega^2)(x-(a_1+a_2+a_3)) &=& x^3-(a_1+a_2+a_3)x^2 + \omega^2 x - (a_1+a_2+a_3)\omega^2 \\
\det(xI-J) &=& x^3-(a_1+a_2+a_3)x^2 + (a_1a_2+a_1a_3+a_2a_3) x  \\
&& \qquad - (a_1a_2a_3+b_1b_2b_3)
\eeqn
Therefore
\begin{equation}
\label{E:omega_squared}
\omega^2 = (a_1a_2+a_1a_3+a_2a_3) = \frac{a_1a_2a_3+b_1b_2b_3}{a_1+a_2+a_3}
\end{equation}
(We show that $a_1+a_2+a_3 \neq 0$ in Remark~\ref{R:marg}.)
Necessary and sufficient conditions for Hopf bifurcation are therefore:
\begin{eqnarray}
\label{e:prodeq1} a_1a_2+a_1a_3+a_2a_3 &>& 0 \\
\label{e:prodeq2}  (a_1+a_2)(a_1+a_3)(a_2+a_3)&=&  b_1b_2b_3
\end{eqnarray}
because 
\[
(a_1a_2+a_1a_3+a_2a_3)(a_1+a_2+a_3)-a_1a_2a_3 =  (a_1+a_2)(a_1+a_3)(a_2+a_3)
\]
Later we also consider the condition
\begin{equation}
\label{E:first}
a_1+a_2+a_3 < 0
\end{equation}
which ensures that
the Hopf branch is the first local bifurcation; 
this is a necessary
 condition for it to be stable. However, this condition need not be sufficient for stability,
 since this depends on higher-order terms in the ODE \cite{GSS88}.

\begin{remark}\rm
\label{R:marg}
The `marginal case' $a_1+a_2+a_3 = 0$ implies that
\[
0 = (a_1+a_2+a_3)^2 = a_1^2+a_2^2+a_3^2 + 2(a_1a_2+a_1a_3+a_2a_3)
\]
contradicting $a_1a_2+a_1a_3+a_2a_3 > 0$. So $a_1+a_2+a_3 \neq 0$.

%This seems to be another proof that no Hopf/steady-state mode interaction is possible
%for this network.
\end{remark}

\begin{remark}\rm
\label{R:prodidentity}
The characteristic equation (multiplied by $-1$) can also be written as
\begin{equation}
\label{e:chareeq}
(x-a_1)(x-a_2)(x-a_3)-b_1b_2b_3 = 0
\end{equation}
Since $\ii \omega$ satisfies the characteristic equation,
\begin{equation}
\label{e:omegaprod}
(\ii \omega - a_1)(\ii \omega - a_2)(\ii \omega - a_3) = b_1b_2b_3
\end{equation}
\end{remark}

\noindent
This implies:
\begin{proposition}
\label{p:bneq0}
The $b_j$ are nonzero.
 \end{proposition}
 
 \begin{proof}
By~\eqref{e:omegaprod}, if any $b_j$ is zero, $\ii \omega$ equals some $a_j$,which is real.
This contradicts $\omega > 0$.
\end{proof}

This proposition is used later when we divide by $b_j$.

Equation~\eqref{e:chareeq} holds when $x = a_1+a_2+a_3$,
since this is also an eigenvalue by \eqref{E:trace}, 
which also implies~\eqref{e:prodeq2}.

\subsection{Eigenstructure}

We describe the eigenstructure of $J$. Let
\[
\tau = a_1+a_2+a_3
\]

\begin{theorem}
\label{T:eigenstructure_J}
The eigenvalues of $J$
are the real eigenvalue $\tau = a_1+a_2+a_3$ and the imaginary pair
$\pm\ii \omega$, where 
\begin{equation}
\label{E:omega_eq}
\omega = \sqrt{a_1a_2+a_1a_3+a_2a_3}
\end{equation}
where we take the positive square root to make $\omega > 0$.

\begin{enumerate}
\item[\rm (1)]
The real eigenvalue $\tau$ has eigenvector 
\[
[1, \frac{\tau-a_1}{b_1}, \frac{\tau-a_1}{b_1}\ \frac{\tau-a_2}{b_2}]^\mathrm{T}
\]
\item[\rm (2)]
The imaginary eigenvalue $ \ii \omega$ has eigenvector
\begin{equation}
\label{e:iomegaeigenvector}
[1, \frac{\ii \omega-a_1}{b_1}, \frac{\ii \omega-a_1}{b_1}\ \frac{\mathrm{i}\omega-a_2}{b_2}]^\mathrm{T}
\end{equation}
\item[\rm (3)]
The imaginary eigenvalue
$-\ii \omega$ has the complex conjugate eigenvector
\[
[1, \frac{-\mathrm{i}\omega-a_1}{b_1}, \frac{-\ii \omega-a_1}{b_1}\ \frac{-\ii \omega-a_2}{b_2}]^\mathrm{T}
\]
\end{enumerate}
\end{theorem}

\proof
Equation \eqref{E:omega_eq} follows from \eqref{E:omega_squared}.

To find the eigenvector for $\tau$, solve
\[
J\Matrix{1\\y\\z} = \tau\Matrix{1\\y\\z}
\]
That is,
\beqn
a_1+b_1y &=& \tau \\
a_2y+b_2z&=& \tau y \\
b_3+ca_3z &=& \tau z 
\eeqn
Since $b_j \neq 0$ by By Proposition \ref{p:bneq0}, 
\beqn
y &=& \frac{\tau-a_1}{b_1} \\
z &=& \frac{b_3}{\tau-a_3} = \frac{\tau-a_1}{b_1}\ \frac{\tau-a_2}{b_2}
\eeqn
using~\eqref{e:prodeq2}.

A similar calculation applies to $\ii \omega$ using~\eqref{e:omegaprod}. 
Then take the complex conjugate.
\qed

Obviously all three eigenvalues are simple, so one of the main hypotheses
of the classical Hopf Bifurcation Theorem holds. 
The other nondegeneracy hypothesis is the eigenvalue crossing condition,
which depends on how $\lambda$ occurs in the ODE.

We now focus on the imaginary eigenvalue $\ii\omega$, because 
its eigenvector controls the phase shifts in the linearised eigenfunction.

\subsection{Scaling Time}
\label{S:ST}

It is possible to scale time $t$ so that $\omega = 1$.
To do so, let $s= \omega t$. Then
\[
\frac{\dd x}{\dd s} = \frac{1}{\omega} \frac{\dd x}{\dd t} = \frac{1}{\omega}f(x)
\]
and the eigenvalues of $\frac{1}{\omega}\mathrm{D}_xf(x_0)$
are $\frac{1}{\omega}$ times those of $\mathrm{D}_xf(x_0)$; that is,
$\frac{\tau}{\omega}, \ii$, and $-\ii$. Renaming $\frac{\tau}{\omega}$ as $\tau$
we may assume that $\omega = 1$.

The period at the Hopf point is then $2\pi$. The group $\Sone = \R/T\Z$
of time translations can then be identified with the unit circle in the  complex
plane, with translation by $\theta$ (mod $T$) identified with $\ee^{\ii \theta}$.

In the sequel we sometimes work with an arbitrary $\omega$,
but when it is more convenient we normalise $\omega$ to $1$ in this manner.

\subsection{Phase Shifts}

We recall our convention for the term `phase shift', relating
waveforms that are identical except for a time-translation.
Here we scale the period to $2\pi$.

\begin{definition}\em
\label{D:phaseshift}
Suppose that $\rho(t)$ is a $2\pi$-periodic function. Let $\theta \in \Sone = \R/2\pi\Z$. Then 
\begin{equation}
\label{E:phase_shift}
\rho(t - \theta)
\end{equation}
is $\rho(t)$ {\em phase-shifted by} $\theta$. 
Also, $\theta$ is the
{\em phase shift from} $\rho(t)$ {\em to} $\rho(t-\theta)$.
To obtain a unique value it is convenient to normalise $\theta$ to lie in $[0,2\pi)$.
\end{definition}

In this definition $\rho(t)$ and $\rho(t-\theta)$ have identical waveforms
except for the phase shift.
In some circumstances the term `phase shift' can meaningfully
be applied to waveforms that are
not identical, for example by comparing the times at which the signals
take their maximum value (provided that this is unique modulo the period).
In particular, near a Hopf bifurcation point we can derive
asymptotic phase (and amplitude) relations, where `asymptotic'
refers to convergence towards the bifurcation point \cite{LG04,LG06}. This is done by
considering the phases and relative amplitudes of the {\em linearised eigenfunctions},
also called {\em modes} in the physics and engineering communities. 
These are linear combinations of the trigonometric functions $\sin t$ and $\cos t$,
defined on the critical eigenspace, and can always be written as
$a \cos(t+\theta)$ with $a > 0$ and $\theta \in \Sone$.
Here $a$ is the {\em amplitude}; with our sign convention, 
$-\theta$ is the {\em relative phase}
(compared to $\cos t$).

Near a Hopf bifurcation point with eigenvalues $\pm \ii$ and linearised
eigenfunction $a \cos(t+\theta)$ with $a > 0$ and $\theta \in \Sone$,
the periodic orbit guaranteed by the Hopf Bifurcation Theorem has:

\quad\quad period near $2\pi$,

\quad\quad amplitude near $a$,

\quad\quad relative phase near $-\theta$.

\noindent
Here `near' means convergence as the bifurcation 
parameter tends to the bifurcation point.
The (complex) linearised eigenfunction takes the form
\[
\phi(t) = e^{\ii t} u
\]
where $u$ is an eigenvector for the eigenvalue $\ii$.
The real eigenfunctions are the real and imaginary parts of this.

Suppose that in polar coordinates
\[
u_j = r_je^{{\ii}\psi_j}
\quad
\mbox{so that}
\quad
r_j = |u_j| \qquad
\psi_j = \arg u_j
\]
with the convention that $0 \leq \arg z < 2\pi$. Then
\[
\phi_j(t) =e^{\ii t} u_j = r_je^{\ii(\psi_j+ t)}
\]

\begin{lemma}
The phase shift from $\phi_j(t)$  to $\phi_k(t)$ is 
\begin{equation}
\label{e:j-k_arg}
\psi_j-\psi_k = 2\pi - \arg\frac{\psi_k}{\psi_j}
\end{equation}
\end{lemma}
\begin{proof}
By Definition~\ref{D:phaseshift} the phase shift from $\phi_j(t)$  to
$\phi_k(t)$ is $\psi_j-\psi_k$. Now
\beqn
\phi_j(t) &=& e^{\ii  t} u_j = r_je^{\ii (\psi_j+ t)} \\
\phi_k(t) &=& e^{\ii  t} u_k = r_ke^{\ii (\psi_k+ t)}
\eeqn
Therefore
\[
\frac{\psi_k}{\psi_j} = \frac{r_k}{r_j}e^{\ii  (\psi_k-\psi_j)}
\]
Therefore
\[
\psi_k-\psi_j = \arg \frac{\psi_k}{\psi_j} \qquad \mbox{so}\qquad \psi_j-\psi_k = 2\pi-\arg \frac{\psi_k}{\psi_j}
\]
proving~\eqref{e:j-k_arg}.
\end{proof}

To compute the phase shifts from once node to the next for Hopf
bifurcation  in the 3-node ring, we need the ratios of successive 
entries in the eigenvector. These are given by:

\begin{lemma}
\label{L:ratios}
A vector $u=[u_1,u_2,u_3]^\mathrm{T}$ is an eigenvector for eigenvalue $\ii$
if and only if all $u_j \neq 0$ and $u$
satisfies the symmetric system of equations
\[
\frac{u_{j+1}}{u_j} = \frac{\ii -a_j}{b_j} \quad j=1,2,3 \pmod{3}
\]
\end{lemma}
\begin{proof}
This follows directly from Theorem \ref{T:eigenstructure_J} with $\omega = 1$.
\end{proof}

\begin{proposition}
\label{p:arg}
Assume that $a_1a_2+a_1a_3+a_2a_3>0$, and normalise $t$ so that
there is a Hopf bifurcation
with $a_1a_2+a_1a_3+a_2a_3=-1$.
Let $\phi(t) = [\phi_1(t), \phi_2(t), \phi_3(t)]^\mathrm{T}$ be the linearised eigenfunction for eigenvalue $\ii $.
Then taking $j \pmod{3}$, the phase shifts $\theta_j$ between $\phi_j(t)$ and $\phi_{j+1}(t)$ are:
\begin{equation}
\label{e:arg}
\theta_j = 2\pi -\arg\left(\frac{\ii -a_j}{b_j} \right) \quad j=1,2,3 \pmod{3}
\end{equation}
Here, when working {\rm (mod 3)}, we replace $0$ by $3$.
\end{proposition}
\proof
By Proposition~\ref{p:bneq0} the 
$b_j$ are nonzero, so we can divide by them.
The general linearised solution has the form
\[
X(t) = \exp(J)X(0)
\]
Take $X(0)= u$ in the (complex) eigenspace for $\ii $. 
Now
\[
X(t) = [e^{\ii  t} u_1, e^{\ii  t} u_2, e^{\ii  t} u_3]^\mathrm{T}
\]
Choosing $u$ as in~\eqref{e:iomegaeigenvector},
there exist real $r_j > 0$ such that
\[
\frac{\ii -a_j}{b_j} = r_j e^{\ii \theta_j} \quad j=1,2,3 \pmod{3}
\]
Now
\[
u_1 = 1 \qquad u_2 = r_1 e^{\ii \theta_1} \qquad u_3 = r_1r_2 e^{\ii (\theta_1+\theta_2)}
\]
By~\eqref{e:j-k_arg} the relative phases are $2\pi$ minus
the arguments of $u_{j+1}/u_j$, which
are therefore the $\theta_j$. (When $j=3$, equation \eqref{e:omegaprod} implies that 
$\theta_1+\theta_2+\theta_3 \equiv 0 \pmod{2\pi}$.
Therefore $e^{\ii (\theta_1+\theta_2+\theta_3)} = 1$.)
\qed

We can rewrite \eqref{e:arg} as
\[
\theta = \arg\frac{-\ii-a_j}{b_j}
\]

\paragraph{Quadrants}

The first, second, third, and fourth {\em open quadrants} consist 
of the complex numbers whose arguments are respectively in the ranges
\[
(0,\pi/2) \qquad (\pi/2,\pi) \qquad (\pi,3\pi/2) \qquad (3\pi/2, 2\pi)
\]

\begin{lemma}
The expression
\[
\frac{\ii -a_j}{b_j}
\]
lies in the:

\quad first quadrant if and only if $a_j < 0, b_j > 0$

\quad second quadrant if and only if $a_j > 0, b_j > 0$

\quad third quadrant if and only if $a_j < 0, b_j < 0$

\quad fourth quadrant if and only if $a_j > 0, b_j < 0$ 
\end{lemma}

\proof
Consider signs of real and imaginary parts.
\qed

\paragraph{Conditions for a Stable Hopf Bifurcation to be Possible}

Necessary conditions to obtain a stable Hopf bifurcation are
that:
\beqn
\tau = a_1+a_2+a_3&<& 0 \\
a_1a_2+a_1a_3+a_2a_3 &>& 0
\eeqn
Because we have Hopf bifurcation at a simple eigenvalue, 
stability then depends (only) on the direction of branching,
determined by non-linear terms. A supercritical branch is stable,
a subcritical one is unstable \cite[Chapter 1 Section 4]{HKW81}. With suitable choices of these nonlinear
terms, we can arrange for a stable branch to exist.

\subsection{Constraints on Signs}

The matrix $J$ has purely imaginary eigenvalues and 
a negative real eigenvalue if and only if
\begin{eqnarray} 
\label{hopf_constraints1}
	(a_1+a_2)(a_1+a_3)(a_2+a_3) &=& b_1b_2b_3 \\
\label{hopf_constraints2}
	a_1 + a_2 + a_3 &<& 0  \\
\label{hopf_constraints3}
	a_1a_2 +a_1a_3 + a_2a_3 &>& 0
\end{eqnarray}

We now show that:
\begin{lemma}
Hopf bifurcation from a stable equilibrium implies
that 
\[
a_1+a_2,\ a_1+a_3,\ a_2+a_3<0
\]
and 
\[
b_1b_2b_3 <0
\]
\end{lemma}

\begin{proof}

Stability of the equilibrium prior to bifurcation is ensured by
\eqref{hopf_constraints2}, Moreover, \eqref{hopf_constraints2} implies that 
at least one of the $a_j < 0$. Renumbering nodes if necessary while
retaining cyclic order (a {\em graph automorphism} \cite{GS16a}) we can assume
$a_1 < 0$. Scaling time by $\delta = 1/|a_1| > 0$  does not affect the signs of
any entries of $J$ and preserves the conditions
(\ref{hopf_constraints1}, \ref{hopf_constraints2}, \ref{hopf_constraints3}).
Therefore without loss of generality we can assume $a_1 = -1$.
Having done so, let
\[
x = a_2 \qquad y = a_3 \qquad c_j = b_j \ (1 \leq j \leq 3)
\]
Now the conditions for Hopf bifurcation from a stable branch of equilibria
when it loses stability
become:
\begin{eqnarray} 
\label{scaled_constraints1}
	(x-1)(y-1)(x+y) &=& c_1c_2c_3 \\
\label{scaled_constraints2}
	x + y &<& 1  \\
\label{scaled_constraints3}
	xy &>& x+y
\end{eqnarray}
We have to prove that each of
\[
(x-1),\ (y-1),\ (x+y) <0
\]
(By Proposition~\ref{p:bneq0}, none of these expressions can be zero.)
First, suppose that $x+y > 0$. Then 
\[
0 < x+y < 1
\]
By~\eqref{scaled_constraints3}, 
\[
(x-1)(y-1) > 1
\]
Therefore $x-1,y-1$ have the same sign.

If $x-1,y-1 > 0$ then $x, y > 1$ so $x+y > 2$, contradicting~\eqref{scaled_constraints2}.

The only possibility remaining is that $x-1,y-1 < 0$, so $x, y < 1$.
By~\eqref{scaled_constraints3}, $xy>0$. So $x, y$ have the same sign. It
cannot be negative because we are assuming $x+y > 0$. Therefore $x, y > 0$.
Now $0 < x, y < 1$, and
\[
x+y-xy > x^2+y^2-xy = (x-y)^2+xy > 0
\]
contradicting~\eqref{scaled_constraints3}.
Therefore $x+y > 0$ is impossible, so $x+y < 0$.

If $x-1, y-1 < 0$, that is, $x, y < 1$, the result is proved. 

If not, we assume at least one of $x-1, y-1 > 0$ and derive
a contradiction as follows.
Since $x+y < 0$, it is not possible for both $x, y > 1$. Therefore 
(interchanging $x, y$ if necessary) we can assume that $x>1, y<1$.
Now $(x-1)(y-1) < 0$, so
\[
xy-x-y = (x-1)(y-1) - 1 < 0
\]
contradicting~\eqref{scaled_constraints3}.
This proves that $a_1+a_2,\ a_1+a_3,\ a_2+a_3<0$
as claimed. Finally, by~\eqref{scaled_constraints1}, $b_1b_2b_3 <0$.
\end{proof}

There are therefore two distinct cases: all three $b_j < 0$,
or some $b_j < 0$ and the other two $> 0$. Up to a graph automorphism,
the second case has $b_1 < 0, b_2 > 0, b_3 > 0$.

\begin{remark} \rm
\label{r:signs}
We have not proved that all three of the $a_j$ are negative,
and examples show that this need not be the case. For example,
the necessary conditions for stable Hopf bifurcation are satisfied when
\[
a_1 = 1 \quad a_2 =  -2 \quad a_3 = -3 \qquad b_1= 1 \quad b_2 = 1 \quad b_3= -10
\]
with eigenvalues $-4,\ii, -\ii$.

It is also possible for one $a_j$ to be zero. An example is
\[
a_1 = 0 \quad a_2 =  -2 \quad a_3 = -3 \qquad b_1= 1 \quad b_2 = 1 \quad b_3= -30
\]
with eigenvalues $-5, \ii\sqrt{6}, -\ii\sqrt{6}$.

If two of the $a_j$ are 0 then $\omega = 0$, which is not permitted.

The signs of the $b_j$ do not affect the existence of the imaginary
eigenvalue or the stability of the real eigenvalue. The first example
above can be modified so that $b_1 = -1, b_2 = -1, b_3 = -10$
and the second can be modified so that $b_1 = -1, b_2 = -1, b_3 = -30$
without affecting the eigenvalues.
\end{remark}

\subsection{Constraints on Phase Shifts}

Equation~\eqref{e:arg} governs the phase shifts $\theta_j$ for
the linearised eigenfunction at Hopf bifurcation, which are the
asymptotic values of phase shifts along the branch 
as it approaches the bifurcation point.

The argument of a complex number is unchanged if it is multiplied by
any positive real number, and is increased by $\pi$ if it is multiplied by
any negative real number. Equation~\eqref{e:arg} shows that
$\theta_j$ depends on $a_1, a_2, a_3$, but only on the sign of $b_j$.

Specifically,
\[
\theta_j = 2\pi-\arg\left(\frac{ \ii \omega - a_j}{b_j}\right) = 2\pi-\arg\left({\rm sgn}(b_j)(\ii\omega - a_j)\right)
\]
So we get either  
\[ 
\theta_j = 2\pi- \arg( \ii\omega - a_j)  \quad \mbox{or}\quad  \theta_j =3\pi- \arg( \ii\omega - a_j)  
\]
depending on the sign of $b_j$, where now $\omega = \pm 1$.

The quadrant to which $\theta_j$ belongs depends only on
the signs of the $a_j, b_j$ and the sign of $\omega$. The value of $\theta_j$ 
depends only on $a_j$ and the sign of  $b_j$.

The main role of the $b_j$ is to affect the amplitudes 
of the three components of the linearised eigenfunction.

To understand the implications of these inequalities,
and relate them to the phase shifts, we tabulate the quadrants
in which the phase shifts $\theta_j$ lie, for all possible combinations
of signs. To reduce the list, we again apply a suitable
relabelling of the nodes in the same cyclic order.

To deal with both imaginary eigenvalues we
must consider the cases $\omega =1$ and $\omega =-1$.
For each such choice we
distinguish three cases A, B, C as listed below.
%\footnote{One $a_j$ can be zero, and then we need
%to redefine `quadrant' to include multiples of $\pi/2$. For the moment
%I'm avoiding that by treating this as a special case.}

\vspace{.2in}
\noindent {\bf Case 1:} $\boldsymbol{\omega =1}$

{\em Case A}: $a_1,a_2,a_3 < 0$.

{\em Case B}: Some $a_j > 0$. Without loss of generality, by relabelling the nodes 
in the same cyclic order, $a_1> 0$.

{\em Case C}: Some $a_j = 0$. Without loss of generality,$a_1 = 0$.

In Case B, we have
\beqn
a_2+a_3 < -a_1 &<& 0\\
a_1(a_2+a_3)+a_2a_3 &>& 0
\eeqn
Therefore
\[
a_2a_3 > - a_1(a_2+a_3) > 0
\]
We cannot have $a_2, a_3 > 0$ since then $a_1+a_2+a_3 > 0$.
Since $a_2a_3> 0$ we must have $a_2, a_3 < 0$.

In Case C, $a_2+a_3 < 0$ and $a_2a_3 > 0$. So
$a_2, a_3 < 0$.

The corresponding quadrants for $(\theta_1, \theta_2, \theta_3)$ are then
as shown in Tables~\ref{T:phaseshiftsA}, ~\ref{T:phaseshiftsB},
and ~\ref{T:phaseshiftsC}.

By Remark~\ref{r:signs}, each of these cases can occur with
either one or three of the $b_j < 0$.

\begin{table}[!htb]
\begin{center}
\begin{tabular}{|ccc|ccc|}
\hline
 &  &  &  quadrant  &  quadrant &  quadrant \\
$b_1$ & $b_2$ & $b_3$ &  $\theta_1$ & $\theta_2$& $\theta_3$\\
\hline
\hline
 - & - & - & 3 & 3 & 3 \\
 - & + & + &  3 & 1 & 1\\
+ & - & + &  1 & 3 & 1 \\
+ & + & - & 1 & 1 & 3 \\
\hline
\end{tabular}
\caption{Classification of combinations of phase shifts by quadrant
in Case A when $\omega > 0$.}
\label{T:phaseshiftsA}
\end{center}
\end{table}

\begin{table}[!htb]
\begin{center}
\begin{tabular}{|ccc|ccc|}
\hline
 &  &  &  quadrant  &  quadrant &  quadrant \\
$b_1$ & $b_2$ & $b_3$ &  $\theta_1$ & $\theta_2$& $\theta_3$\\
\hline
\hline
 - & - & - & 4 & 3 & 3 \\
 - & + & + &  4 & 1 & 1\\
+ & - & + &  2 & 3 & 1 \\
+ & + & - & 2 & 3 & 1 \\
\hline
\end{tabular}
\caption{Classification of combinations of phase shifts by quadrant
in Case B when $\omega > 0$.}
\label{T:phaseshiftsB}
\end{center}
\end{table}

\begin{table}[!htb]
\begin{center}
\begin{tabular}{|ccc|ccc|}
\hline
 &  &  &  angle  &  quadrant &  quadrant \\
$b_1$ & $b_2$ & $b_3$ &  $\theta_1$ & $\theta_2$& $\theta_3$\\
\hline
\hline
 - & - & - & $3\pi/2 $& 3 & 3 \\
 - & + & + &  $3\pi/2$ & 1 & 1\\
+ & - & + &  $\pi/2$ &3 & 1 \\
+ & + & - & $\pi/2 $ & 3 & 1 \\
\hline
\end{tabular}
\caption{Classification of combinations of phase shifts by quadrant
in Case C when $\omega > 0$. Note that $\theta_1$ is specified exactly.}
\label{T:phaseshiftsC}
\end{center}
\end{table}

Columns 2 and 3 are the same in all three tables,
since only the sign of $a_1$ changes. Table C is the transitional case.

This motivates:
\begin{definition}\em
The periodic solution is a {\em rotating wave} or
{\em maximally asynchronous} if all $\theta_j$
lie in quadrant 2 or all  $\theta_j$
lie in quadrant 3 (for a reverse direction wave).

It is a {\em standing wave} or
{\em maximally synchronous} if it is not a rotating wave.
\end{definition}

\begin{theorem}
The solution is a rotating wave if and only if all $a_j, b_j$ are negative.
\end{theorem}
   
\begin{proof}
Inspect the three tables. 
\end{proof}

In all other cases at least two nodes have `absolute' (minimal positive) phase difference
$\pi/2$ or less (quadrants 1, 4). Indeed, there is always something in quadrant 1.

\vspace{.2in}
\noindent {\bf Case 2:}  $\boldsymbol{\omega =-1}$

The results in this case can be obtained from those for
$\omega =1$ by taking the complex conjugate (equivalently,
replacing $a_j$ by $-a_j$). This changes the quadrants
by interchanging 1 with 4 and 2 with 3.
The corresponding quadrants for $(\theta_1, \theta_2, \theta_3)$ are then
as shown in Tables~\ref{T:phaseshiftsAneg}, ~\ref{T:phaseshiftsBneg},
and ~\ref{T:phaseshiftsCneg}.

\begin{table}[!htb]
\begin{center}
\begin{tabular}{|ccc|ccc|}
\hline
 &  &  &  quadrant  &  quadrant &  quadrant \\
$b_1$ & $b_2$ & $b_3$ &  $\theta_1$ & $\theta_2$& $\theta_3$\\
\hline
\hline
 - & - & - & 2 & 2 & 2 \\
 - & + & + &  2 & 4 & 4\\
+ & - & + &  4 & 2 & 4 \\
+ & + & - & 4 & 4 & 2 \\
\hline
\end{tabular}
\caption{Classification of combinations of phase shifts by quadrant
in Case A when $\omega < 0$.}
\label{T:phaseshiftsAneg}
\end{center}
\end{table}

\begin{table}[!htb]
\begin{center}
\begin{tabular}{|ccc|ccc|}
\hline
 &  &  &  quadrant  &  quadrant &  quadrant \\
$b_1$ & $b_2$ & $b_3$ &  $\theta_1$ & $\theta_2$& $\theta_3$\\
\hline
\hline
 - & - & - & 1 & 2 & 2 \\
 - & + & + &  1 & 4 & 4\\
+ & - & + &  3 & 2 & 4 \\
+ & + & - & 3 & 2 & 4 \\
\hline
\end{tabular}
\caption{Classification of combinations of phase shifts by quadrant
in Case B when $\omega < 0$.}
\label{T:phaseshiftsBneg}
\end{center}
\end{table}

\begin{table}[!htb]
\begin{center}
\begin{tabular}{|ccc|ccc|}
\hline
 &  &  &  angle  &  quadrant &  quadrant \\
$b_1$ & $b_2$ & $b_3$ &  $\theta_1$ & $\theta_2$& $\theta_3$\\
\hline
\hline
 - & - & - & $\pi/2 $& 2 & 2 \\
 - & + & + &  $\pi/2$ & 4 & 4\\
+ & - & + &  $3\pi/2$ & 2 & 4 \\
+ & + & - & $3\pi/2 $ & 2 & 4 \\
\hline
\hline
\end{tabular}
\caption{Classification of combinations of phase shifts by quadrant
in Case C when $\omega < 0$. Note that $\theta_1$ is specified exactly.}
\label{T:phaseshiftsCneg}
\end{center}
\end{table}

\section{Directed Rings of $\boldsymbol n$ nodes}
\label{S:DRnN}

Much of the above generalises to a directed ring of $n$ nodes
with nearest-neighbour coupling. Now the Jacobian takes the form
\[
J = \Matrix{a_1 & b_1 & 0 & \cdots & 0 \\
	 0 & a_2 & b_2 & \cdots & 0 \\
	\vdots & 0 & \ddots & \cdots & \vdots \\
	b_n & 0 & \cdots & 0 & a_n}
\]

A version of the following theorem is proved in
\cite[Theorem SM1.1, supplementary material]{GGPSW19} for any fully inhomogeneous network,
but a direct proof is straightforward and yields slightly more information,
as discussed after the proof.

\begin{theorem}
\label{T:simple_generic}
The eigenvalues of $J$ are generically simple. Indeed, for any 
fixed $a_1, \ldots, a_n$, small perturbations of $b_1, \ldots, b_n$
remove any multiple eigenvalues.
\end{theorem}
\begin{proof}
Expanding $\det(J)$ along row 1 and using induction, we see that
\[
\det(J) = a_1 \cdots a_n + (-1)^{n+1} b_1 \cdots b_n
\]
Replacing $a_i$ by $a_i-\lambda$, the characteristic polynomial of $J$ is
therefore
\[
p(\lambda) = \det(J-\lambda I) = (a_1-\lambda) \cdots (a_n-\lambda) + (-1)^{n+1} b_1 \cdots b_n
\]
This has a multiple zero at some $\lambda$ if and only if
\begin{equation}
\label{E:mult_zero}
p(\lambda) = p'(\lambda) = 0
\end{equation}
Now
\beqn
p'(\lambda) &=& (-1)[
(a_2-\lambda) (a_3-\lambda) \cdots (a_n-\lambda) \\
&& + (a_1-\lambda) (a_3-\lambda)\cdots (a_n-\lambda) \\
&& + \cdots + (a_1-\lambda) (a_3-\lambda)\cdots (a_{n-1}-\lambda)
]
\eeqn
which is independent of the $b_i$.

Let the zeros of $p'(\lambda)$ be $\lambda_1, \ldots, \lambda_{n-1}$.
Perturb the $b_i$ to make them all nonzero. Then perturb $b_1$ (say) again
to ensure that
\[
b_1 \cdots b_n \neq p(\lambda_i) \quad 1 \leq i \leq n-1
\]
which can always be done since $\{p(\lambda_i):1 \leq i \leq n-1\}$
is finite and independent of the $b_i$. Now condition \eqref{E:mult_zero} is false.
\end{proof}

This proof shows that for any $a_i$ (linearised internal dynamic)
we can obtain simple eigenvalues by perturbing only the $b_i$
(linearised couplings.)

Assume $\mu$ is an eigenvalue
(later we make $\mu = \ii \omega$ to get a Hopf bifurcation point),
so that
\[
\det(J-\mu I) = 0
\]
Let the corresponding eigenvector be
\[
u = [u_1, u_2, \ldots, u_n]^\mathrm{T}
\]
Then:
\beqn
a_1u_1+b_1u_2 &=& \mu u_1 \\
a_2u_2+b_2u_3 &=& \mu u_2 \\
& \vdots &\\
a_ju_j+b_ju_{j+1} &=& \mu u_j \\
& \vdots &\\
a_{n-1}u_{n-1}+b_{n-1}u_n &=& \mu u_{n-1} \\
b_nu_1+a_nu_n &=&  \mu u_n
\eeqn
Thus for $1 \leq j \leq n-1$ we have
\[
u_{j+1} = \frac{\mu - a_j}{b_j}\,u_j
\]
and $u_1$ is arbitrary. We have $u_1 \neq 0$ or else all $u_j = 0$.
Explicitly:
\beqn
u_1 &=& 1 \\
u_2 &=& \frac{\mu - a_1}{b_1} \\
u_3 &=& \frac{\mu - a_1}{b_1}\,\frac{\mu - a_2}{b_2} \\
u_4 &=& \frac{\mu - a_1}{b_1}\,\frac{\mu - a_2}{b_2}\,\frac{\mu - a_3}{b_3} \\
&\vdots & \\
u_n &=& \frac{\mu - a_1}{b_1} \cdots \frac{\mu - a_{n-1}}{b_{n-1}} 
\eeqn

\begin{remark}\rm
Since $\mu$ satisfies the characteristic equation, 
\[
0 = \det (J-\mu I) = (a_1-\mu) \cdots (a_n-\mu) + (-1)^{n+1} b_1 \cdots b_n
\]
because the $n$-cycle $(12\ldots n)$ is an even or odd permutation according as
$n$ is odd or even. Therefore
\[
(\mu-a_1) \cdots (\mu-a_n) =  b_1 \cdots b_n
\]
Further multiplication of $u_n$ by $\frac{\mu - a_n}{b_n}$
gets back to $u_1 = 1$.
\end{remark}

We now get a direct generalisation of Proposition~\ref{p:arg}
and equation~\eqref{e:arg}:

\begin{proposition}
\label{p:arg-n}
Assume that, after scaling time, there is a Hopf bifurcation
with eigenvalue $\ii$.

Let $\phi(t) = [\phi_1(t), \ldots \phi_n(t)]^\mathrm{T}$ be the linearised eigenfunction for eigenvalue $\ii \omega$.
Then taking $j \pmod{n}$, the phase shifts $\theta_j$ between $\phi_j(t)$ and $\phi_{j+1}(t)$ are:
\begin{equation}
\label{e:arg-n}
\theta_j = 2\pi-\arg\left(\frac{\ii -a_j}{b_j} \right) \quad j=1, \ldots n \pmod{n}
\end{equation}
(Here, when working$\! \pmod{n}$, we replace $0$ by $n$.)
\end{proposition}

\begin{theorem}
\label{T:nores}
Generically there are no resonances (that is,
$\ii$ and $k \ii$ both being eigenvalues, for rational
$k \neq \pm 1$). Indeed, for any 
fixed $a_1, \ldots, a_n$, small perturbations of $b_1, \ldots, b_n$
remove any $p:q$ resonances among eigenvalues.
\end{theorem}

\begin{proof}
The idea is to convert a resonance into a double eigenvalue for a
related polynomial and then derive a contradiction. 
Let $k  \neq \pm 1$, and let
\beqn
A(\lambda) &=& (a_1-\lambda) \cdots (a_n-\lambda) \\
c &=& (-1)^{n+1} b_1 \cdots b_n
\eeqn
Then the characteristic polynomial $p(\lambda) = A(\lambda)+c$.
Suppose that at some $\lambda = \lambda_0$ we have $p(\lambda_0) = p(k \lambda_0) = 0$
where $k \in \mathbb{N}, k \geq 2$. Then
\[
A(\lambda_0)+c = A(k\lambda_0) + c = 0
\]
Now $p(\lambda)p(k\lambda)$ has a double zero at $\lambda_0$,
so its derivative vanishes there. That is,
\[
p'(\lambda)p(k\lambda)+kp(\lambda)p'(k\lambda) = 0
\]
at $\lambda_0$, so
\[
p'(\lambda)(A(k\lambda)+c) + k(A(\lambda)+c)p'(k\lambda) = 0
\]
at $\lambda_0$. Thus
\[
p'(\lambda)A(k\lambda) + kA(\lambda)p'(k\lambda) +c(p'(\lambda)+kp'(k\lambda))= 0
\]
at $\lambda_0$. Also $c = -A(\lambda)$ at $\lambda_0$.
Substituting,
\[
Q(\lambda) = p'(\lambda)(A(k\lambda)- A(\lambda)) - p'(k\lambda)(k-1)A(\lambda) = 0
\]
at $\lambda_0$.

The polynomial $Q$ is independent of $c$ (that is, of the $b_i$).
We claim that since $k \geq 2$ it does not 
vanish identically. To see why, consider the highest-order terms.
Since $p(\lambda) = A(\lambda)+c$, these terms are
\beqn
&&(n\lambda^{n-1})(k^n\lambda^n)-(n\lambda^{n-1})(\lambda^n)
	-(nk^{n-1}\lambda^{n-1})(\lambda^n)(k-1)\\
&&	= \lambda^{2n-1}(-nk^{n-1}+n) = \lambda^{2n-1}n(1-k^{n-1})
\eeqn
The coefficient $n(1-k^{n-1})$ vanishes only when $k=1$, or $k=-1$
when $n$ is odd, but $k \neq \pm 1$.
So $Q$ does not vanish identically, as claimed. 

Therefore $Q$ has at most $2n-1$ zeros in $\C$.
Let the zeros be $\lambda_1 , \ldots, \lambda_{2n-1}$. The value of $\lambda_0$
must be one of the $\lambda_i$ for $1 \leq i \leq 2n-1$.
Now, for any choice of $c$, the equation
\[
c= -A(\lambda_i) 
\]
holds for some $i$. But the number of values of $A(\lambda_i)$ is finite,
so a small perturbation of the $b_i$ removes the double zero,
that is, this particular resonance.

Any finite (indeed, countable) 
number of small perturbations can be chosen so that they
combine to give a small perturbation, since $\eps/2+\eps/4+ \cdots = \eps$.
This result can therefore be used to remove all nontrivial resonances ($k \neq \pm1$).

\end{proof}

\end{document}